\documentclass{article}
\usepackage[affil-it]{authblk}
\usepackage[top=2cm,bottom=2cm,left=2cm,right=2cm]{geometry}
\usepackage[title]{appendix}
\usepackage{amsmath}
\usepackage{comment}
\usepackage{mathrsfs}
\usepackage{amssymb}
\usepackage{color}
\usepackage{graphicx}

\usepackage{amsthm}
\newtheorem{theorem}{Theorem}
\newtheorem{claim}{Claim}

\newtheorem{proposition}{Proposition}

\newcommand{\argmax}{\operatornamewithlimits{argmax}}

\newcommand{\bfP}{\mathbf{P}}
\newcommand{\bfQ}{\mathbf{Q}}
\newcommand{\dham}{\mathbf{d}_H}

\newcommand{\Lmn}{\mathcal{L}_{m,n}}
\newcommand{\N}{\mathcal{N}}
\newcommand{\Pmn}{\mathcal{P}_{m,n,d}^{\epsilon}}
\newcommand{\Pz}{\mathcal{P}_{m,n,0}^{\epsilon}}
\newcommand{\tr}{\mathbf{tr}}
\begin{document}
\title{Dynamics in atomic signaling games}
\author{Michael J. Fox, Behrouz Touri, and Jeff S. Shamma \\\small{School of Electrical and Computer Engineering\\ Georgia Institute of Technology}}
\date{}
\maketitle
\begin{abstract}
We study an atomic signaling game under stochastic evolutionary dynamics. There is a finite number of players who repeatedly update from a finite number of available languages/signaling strategies. Players imitate the most fit agents with high probability or mutate with low probability.
We analyze the long-run distribution of states and show that, for sufficiently small mutation probability, its support is limited to efficient communication systems. We find that this behavior is insensitive to the particular choice of evolutionary dynamic, a property that is due to the game having a potential structure with a potential function corresponding to average fitness. Consequently, the model supports conclusions similar to those found in the literature on language competition. That is, we show that efficient languages eventually predominate the society while reproducing the empirical phenomenon of linguistic drift.
The emergence of efficiency in the atomic case can be contrasted with results for non-atomic signaling games that establish the non-negligible possibility of convergence, under replicator dynamics, to states of unbounded efficiency loss.



\end{abstract}

\section{Introduction}
\label{intro}
Biological systems at many different scales depend on reliable and efficient signaling. Mathematical modeling of signaling may provide insights into conditions conducive to the emergence of communication in biological \cite{smith2003animal} and non-biological \cite{Yong_Coevolution,Pollack_Linear_PE} settings. A key problem is that of coordination. That is, how do systems develop consistent coordination/communication protocols without the benefit of a centralized coordinating entity?

One way to model the coordination problem in distributed communication is through signaling games \cite{Lewis_Convention}. Researchers began studying these games in a biological context more recently \cite{springerlink:10.1007/s002850070004,NOWAK1999147}. The strategies available to the players in a signaling game are pairs of mappings. A speaking strategy is a mapping from the set of objects to the set of symbols, while a hearing strategy is a mapping from the set of symbols back to the set of objects. A communication event involves two players, i.e., a speaker and hearer, an object, and a signal or a message to communicate the object between the two players. The player assigned the role of speaker produces a signal that her speaking strategy associates with the object. The player assigned the role of hearer then announces the object her hearing strategy associates with this signal. If the interpreted object agrees with the original one, then the communication event is considered to be successful. Such games can be studied in the presence of a network of agents where agents meet randomly (or deterministically) and use their speaking and hearing strategies to communicate with each other.

A fundamental question in signaling games is identification of distributed learning schemes that lead to an efficient communication system. In the continuum agent, or ``non-atomic'', setting it has been shown that the replicator dynamics can converge, from a set of initial conditions with positive measure, to neutrally stable states that do not maximize communication efficiency \cite{Pawlowitsch2008203,Huttegger2007-HUTEAT-2}. While some guarantees on performance still exist \cite{evolang_fox}, we focus our attention here on processes leading to maximum efficiency. Selection-mutation dynamics have been suggested \cite{Hofbauer2008843} as an alternative to explain away the inefficiency. The system is analyzed for the special case of binary signaling games. The ``mass action'' perspective of non-atomic signaling games is taken up for analytical convenience. The more realistic discrete agent, or ``atomic'', model is approximated by the non-atomic model over finite time horizons for sufficiently large populations \cite{Benaim2003}. Characterization of states favored by selection in the frequency dependent Moran process \cite{Pawlowitsch2007606} has been carried out for the non-atomic signaling game.

In this paper we study the long-run behavior of stochastic evolutionary dynamics in the non-atomic signaling game.

As a starting point we show that the non-atomic signaling game is a potential game \cite{MondererShapley96}. Some learning dynamics exist (see for instance \cite{RePEc:eee:gamebe:v:5:y:1993:i:3:p:387-424}) that are equipped with substantive performance guarantees for all or some of the potential games. For such dynamics, the problem of efficiency becomes quite trivial, which is good news for the proactive engineering of communicating agents. In essence, it turns out that in potential games, the individual, myopic and distributed optimization activities of agents leads to centralized optimization of the so-called potential function. Since the potential function of the atomic signaling game is proportional to average fitness, it is intuitively reasonable that dynamics resembling natural selection\footnote{Our results also can be interpreted in the context of cultural evolution, but we emphasize the biological interpretation first and foremost.} with random mutations achieve maximum average fitness.

Rather, our focus is on evolutionary dynamics that are motivated imitation dynamics. Agents randomly imitate successful other agents or, with small probability, mutate. We will show that over the long run, agents mostly coordinate on a single efficient language. The form of that language will change over time consistent with the empirically observed phenomenon of linguistic drift \cite{Jespersen_PL}.

The tools utilized herein are analysis methods for perturbed Markov chains. In particular, we will focus on the notion of \textit{stochastic stability}. Because of persistent random effects, our dynamics need not converge to any specific state. Rather, we analyze the long run occupancy measures of different states. As a certain mutation probability becomes small, agents will spend almost all of the time in agreement on an efficient language. By contrast, all other states are visited with vanishingly small long run frequency. Such methods are well established for the analysis of evolutionary systems, especially for problems of social coordination and evolution of convention \cite{Young_Conventions,Young_ISSS}.

The remainder of this paper is organized as follows. Section~2 presents the framework of signaling games and establishes the potential game property. Section~3 introduces a dynamic process based on imitation with mutations. Section~3 goes on to present the main results on the stochastic stability of efficient languages. {Section~4 presents a variation on the dynamic process that allows for more flexibility in determining which agents are allowed to reproduce}. Section~5 offers some concluding remarks. Finally, background theory on stochastic stability is presented in an appendix.

\section{Signaling Games}

\subsection{Setup}

There are $m \geq 2$ objects and $n \geq 2$ symbols. A speaking strategy is represented by an $m \times n$ binary, row stochastic matrix\footnote{A binary, row-stochastic matrix is a matrix that has one element per row that is equal to one and all other elements equal to zero.} $P$. If $P_{ij}=1$ then the speaking strategy associates object $i$ with signal $j$. Similarly, a hearing strategy is an $n \times m$ binary, row-stochastic matrix $Q$. If $Q_{ij}=1$ then the hearing strategy associates signal $i$ with object $j$. Thus, assuming uniform prior probability over objects that are observed by the speaker, the symmetric payoff of a player utilizing strategy $P$ for speaking and a player utilizing strategy $Q$ for hearing is proportional to
\begin{equation}\nonumber
\sum_i \sum_j P_{ij}Q_{ji} = \tr (PQ),
\end{equation}
where $\tr (M)$ denotes the trace of a square matrix $M$.

We call a joint speaking and hearing strategy $(P,Q)$ a \textit{language} and use $\mathcal{L}_{m,n}$ to refer to the set of all such pairs.
In this paper, we study the atomic signaling game in which there is a finite set of players, with each one selecting her own speaking and hearing strategies. However, we first review the non-atomic version of the game.

\subsection{The non-atomic signaling game}
The set $\mathcal{L}_{m,n}$ has cardinality $m^nn^m$. Assume any ordering on the elements of $\mathcal{L}_{m,n}$ so that the $m^nn^m$-dimensional simplex
\begin{equation}\nonumber
S_{m,n} = \{\textbf{x} \in \mathbb{R}^{m^nn^m}: \sum_i x_i = 1, x_i \geq 0 \quad \forall i\}
\end{equation}
gives the possible distributions of a single population over the set of languages. Let $(P^k,Q^k)$ be the $k$th language in the ordering. Then the fitness of a player utilizing $(P^k,Q^k)$ in a population state \textbf{x} is
\begin{equation}\nonumber
f_k(\textbf{x}) = \sum_{i=1}^{m^nn^m} x_i \left( \tr (P^kQ^i)+\tr (P^iQ^k) \right) .
\end{equation}
In other words, the fitness $f_k(\textbf{x})$ of an agent speaking the $k$th language is the expected payoff of its interaction with a random language that is distributed according to probability distribution $\mathbf{x}$. Our main focus in this paper is to focus on the evolution of the language in a population. In other words, we are interested in the proportion of the society speaking each language $\mathbf{x}(t)$ as a function of time $t$.

In the case of the non-atomic language game, one way to model this evolution is through the replicator dynamics
\begin{equation}\label{eqn:replicator}
\dot{x}_i(t) = x_i(t) \left( f_i(x(t)) - \sum_j x_j(t) f_j(x(t)) \right).
\end{equation}
In what follows, for the sake of notational simplicity, whenever there is no confusion, we may drop the time variable $t$.
Intuitively, the replicator dynamics \eqref{eqn:replicator} means that the rate of increase of the population speaking the $i^{\text{th}}$ language, i.e.\ $\frac{d}{dt} \log(x_i)$, is proportional to relative fitness of $i^{\text{th}}$ language compared to other languages.

One of the issues with the replicator dynamics \eqref{eqn:replicator} for non-atomic populations is that it may converge to states that do not maximize the average fitness,
\begin{equation}\nonumber
W(\textbf{x}) = \sum_{i = 1}^{m^nn^m}x_i f_i(\textbf{x}).
\end{equation}
It has been shown \cite{Pawlowitsch2008203}, \cite{Huttegger2007-HUTEAT-2} that convergence to these sub-optimal states occurs from a set of initial states with non-zero measure.

In order to reconcile this fact with the intuitive notion that evolution leads to efficient signaling, a number of alternative models have been proposed. Mutation-selection dynamics, a perturbation of the replicator dynamics, have been studied in the binary case \cite{Hofbauer2008843}. These dynamics add a ``mutation'' term to the replicator dynamics, intended to capture the effect of random mutations. In non-biological contexts this term can be interpreted as experimentation.

In Section~\ref{secdynamics} we will study a similar dynamic. However, we instead concentrate on the finite-population, or ``atomic'' game. The motivation for considering atomic agents is that it enables us to analyze the long-run behavior of stochastic evolutionary dynamics. A common justification for studying mass-action heuristics like the non-atomic signaling game is that, over finite time horizons, these models approximate stochastic evolutionary dynamics with sufficiently many atomic agents \cite{Benaim2003}.

\subsection{The atomic signaling game}

Consider a society consisting of $N$ agents, $1,\ldots,N$, and suppose that agent $i$ speaks language $(P_i,Q_i)\in \mathcal{L}_{m,n}$. Let $(\textbf{P},\textbf{Q}) \in \mathcal{L}_{m,n}^N$ be a vector consisting of the $N$ languages spoken by the agents. Similarly to the infinite-population model, the fitness of agent $i$ can be defined as the expected payoff of her interacting with a random agent $j$ picked uniformly from $\{1,\ldots,N\}$, i.e.,
\begin{equation}\label{eqn:payoffs}
f_i(\textbf{P},\textbf{Q}) = \tr (P_i \frac{1}{N-1} \sum_{j \neq i} Q_j) + \tr (\frac{1}{N-1}\sum_{j \neq i} P_j  Q_i).
\end{equation}
The frequency dependent Moran process has been analyzed for this game \cite{Pawlowitsch2007606}, suggesting that evolution tends towards efficient states in the limit of so-called ``weak selection''. Recent work has addressed and contrasted the implications of weak selection versus other models \cite{wuetal}. Here, we instead seek to characterize explicitly the long-run behavior of this game under stochastic evolutionary dynamics.

We first show that the atomic signaling game has the underlying structure of being a potential game. An $N$ player game with strategy sets $\mathcal{A}_1,\ldots,\mathcal{A}_N$ and payoff functions $f_1,\ldots,f_N$ is referred to as \textit{a potential game} \cite{MondererShapley96} if there exists a function $\Phi: \mathcal{A}=\mathcal{A}_1\times\cdots\mathcal{A}_N \rightarrow \mathbb{R}$ such that for any player $i$, any joint strategy $\textbf{s}\in\mathcal{A}$, and any strategy $s\in\mathcal{A}_i$ of player $i$ we have
\begin{equation}\nonumber
f_i(s,\textbf{s}_{-i}) - f_i(\textbf{s}) = \Phi(s,\textbf{s}_{-i}) - \Phi(\textbf{s}),
\end{equation}
where $\textbf{s}_{-i}$ is the vector of strategies for players other than $i$.  An important feature of potential games is that often myopic uncoordinated optimization of individual payoffs leads to optimization of the potential function $\Phi$.

 In fact, the atomic language game with payoff functions \eqref{eqn:payoffs} is a potential game.
\begin{theorem}\label{thm:potential}
   The finite-population language game is a potential game with potential function $\Phi \equiv \frac{1}{2}\sum_{i=1}^N f_i$.
\end{theorem}

\begin{proof}
   Let $(\textbf{P},\textbf{Q})$ and $(\hat{\textbf{P}},\hat{\textbf{Q}})$ differ only in the language of player $\hat{i}$. Then
\begin{align*}
2 \Phi(\mathbf{P},\mathbf{Q}) &=  \sum_i f_i(\mathbf{P},\mathbf{Q})\\
&= \sum_i \left(\tr (P_i \frac{1}{N-1} \sum_{j\not= i} Q_j)
+ \tr (\frac{1}{N-1} \sum_{j\not= i} P_j Q_i)\right)\\
&= \sum_{i\not= \hat{i}} \left(\tr (P_i \frac{1}{N-1} \sum_{j\not= i,\hat{i}} Q_j)
+ \tr (\frac{1}{N-1} \sum_{j\not= i,\hat{i}} P_j Q_i)\right)\\
&\quad + 2\left( \tr (P_{\hat{i}} \frac{1}{N-1} \sum_{j\not= \hat{i}} Q_j)
+ \tr (\frac{1}{N-1} \sum_{j\not= \hat{i}} P_j Q_i)\right)\\
&= \sum_{i\not= \hat{i}} \left(\tr (P_i \frac{1}{N-1} \sum_{j\not= i,\hat{i}} Q_j)
+ \tr (\frac{1}{N-1} \sum_{j\not= i,\hat{i}} P_j Q_i)\right)\\
&\quad + 2 f_i(\mathbf{P},\mathbf{Q}).
\end{align*}
A similar expression holds for $\Phi(\hat{\mathbf{P}},\hat{\mathbf{Q}})$. Accordingly,
$$\Phi(\mathbf{P},\mathbf{Q}) - \Phi(\hat{\mathbf{P}},\hat{\mathbf{Q}})
= f_{\hat{i}}(\mathbf{P},\mathbf{Q}) - f_{\hat{i}}(\hat{\mathbf{P}},\hat{\mathbf{Q}}),$$
as required.
 \end{proof}

Since the potential function is proportional to average fitness, it is not surprising that many stochastic evolutionary dynamics on language games tend to maximize average fitness. Indeed, from the perspective of distributed algorithm design problems of this form are well studied and generic procedures with strong performance guarantees exist. For instance, under logit dynamics players spend almost all of their time at maximizers of the potential function over the long run as a temperature parameter is sufficiently close to zero \cite{RePEc:eee:gamebe:v:5:y:1993:i:3:p:387-424}. However, logit bears little resemblance to the replicator dynamics studied in the non-atomic signaling game.

Many variations on logit have been suggested, motivated by concerns such as information and actuation constraints in engineering applications \cite{marden-shamma-08} or behavioral tendencies and rates of convergence \cite{DBLP:conf/sigmetrics/ShahS10}. Our paper is novel in its insistence on replicator-like dynamics (see also \cite{DBLP:KleinbergPT09}). The dynamics we suggest are analyzed only for the atomic signaling game. Characterizing the equilibrium selection properties and rates of convergence of such algorithms more generally is a future direction.

\section{Evolutionary Dynamics}\label{secdynamics}

\subsection{The model}

 As described above, our focus in this paper is the study of the evolution of language across a society. To describe this model, we first present a few definitions and notations.

Define the Hamming distance $\dham((P,Q),(P',Q'))$ between two languages $(P,Q)$ and $(P',Q')$ to be
 \begin{align}\label{eqn:hamming}
   \dham((P,Q),(P',Q'))=\frac{1}{4}\sum_{i=1}^m\sum_{j=1}^n(|P_{ij}-P'_{ij}|+|Q_{ij}-Q'_{ij}|).
 \end{align}
 Accordingly, let $\rho((P,Q),d)$ be the $d$-disk around the language $(P,Q)$
 \begin{align}\label{eqn:ddisk}
   \rho((P,Q),d)=\{(P',Q')\mid \dham((P,Q),(P',Q'))\leq d\},
 \end{align}
 where $d\geq 0$.

Our model for the evolution of the language in a society is based on the evolution of a random dynamical system that evolves at discrete time stages $t=0,1,\ldots$. Let $(\bfP[t],\bfQ[t])=((P_1[t],Q_1[t]),\ldots,(P_N[t],Q_N[t]))$ be the vector consisting of the languages of $N$ players (agents) in the society at time $t$. Our evolutionary model is as follows:

\begin{itemize}
\item At $t\geq 1$, each player $i$, randomly and independently of the other players and earlier choices of all players, chooses to revise her strategy with some probability $p_i\in(0,1)$.

\item If player $i$ is ``active'', i.e., chooses to revise her strategy, she will set
\begin{equation}\label{eqn:maindynamic}
(P_i[t],Q_i[t]) = \begin{cases}
(P_{\hat{k}},Q_{\hat{k}}), & \text{with probability } 1-\epsilon \\
\operatorname{rand}(\rho((P_i[t-1],Q_i[t-1]),d)), & \text{with probability } \epsilon
\end{cases}
\end{equation}
where
$$\hat{k} \in \argmax_k f_k(\textbf{P}[t-1],\textbf{Q}[t-1]);$$
$\operatorname{rand}(\cdot)$ indicates the outcome of uniform random sampling from the given set; and $d\geq 1$ is a fixed parameter. In the case of multiple maximizers, we choose $\hat{k}$ from $\argmax_k f_k(\textbf{P}[t-1],\textbf{Q}[t-1])$ at random uniformly.

\item The agents who choose not to revise their strategies leave their strategies unchanged, i.e.,
\begin{equation}\nonumber
(P_j[t],Q_j[t]) = (P_j[t-1],Q_j[t-1]) \quad \mbox{for all $j$ not being active at time $t$}.
\end{equation}
\end{itemize}

Let $\mathcal{P}_{m,n,d}^{\epsilon}$ denote these evolutionary dynamics. For notational simplicity, the initial condition is suppressed in this notation.

In words, when player $i$ updates her strategy (with probability $p_i$), the player either imitates the fittest (with probability $1-\epsilon$) or mutates to a nearby strategy (with probability $\epsilon$).  The distance $d$ determines how far a strategy can possibly mutate. For sufficiently large $d$, the mutation can be to an arbitrary strategy.


In contrast to replicator dynamics, reproductive opportunities are afforded to only the fittest players as opposed to reproduction in proportion to fitness. Rather, the dynamics reflect the feature of imitation.  Furthermore, unused strategies are not subsequently utilized except through (rare) mutations. We will consider a variant of these dynamics in Section \ref{sec:dynamic_var}, which allows for reproduction among more than the fittest agents. The analysis of those models is a straightforward extension of the results for the present model, and so we study the current model extensively in this section.

\subsection{Stochastically stable states}

The dynamics are random in that at each stage, the language spoken by an agent in the next stage is not specified deterministically, because of both the random decision whether or not to be active as well as the persistent possibility of mutations. Accordingly, we cannot discuss the long run properties in terms of the convergence of the state. Rather, our study focuses on characterizing the \textit{stochastically stable} states of this process (see the appendix). In particular, we will characterize the states that are occupied \textit{almost exclusively} in the long run for small mutation probability, $\epsilon$. On the other hand, states that are \textit{not} stochastically stable are visited with vanishingly small frequency.

The dynamics $\mathcal{P}_{m,n,d}^{\epsilon}$ constitute a time-homogenous Markov chain on the state space $\Lmn^N$.

\begin{proposition} The (Markov chain) dynamics $\mathcal{P}_{m,n,d}^{\epsilon}$ with $\epsilon > 0$ admit a unique stationary distribution.
\end{proposition}
\begin{proof}
Since all revision probabilities satisfy $p_i > 0$, it is possible from any state to transition to any other  state in a finite number of stages.  Furthermore, since all revision probabilities satisfy $p_i < 1$, there is always a positive probability that the state of the Markov chain remains unchanged from one stage to the next. Therefore, such a Markov chain is irreducible and aperiodic and hence, it admits a unique stationary distribution.
\end{proof}

Let $\mu_\epsilon$ denote the stationary distribution under mutation probability $\epsilon$. The vector, $\mu_\epsilon$, is a probability distribution over the \textit{set} $\Lmn^N$. The component $\mu_\epsilon(\mathbf{P},\mathbf{Q})$ denotes the steady state probability of being in the state $(\mathbf{P},\mathbf{Q})\in\Lmn^N$.

A state $(\mathbf{P},\mathbf{Q})$ is stochastically stable if
$$\lim_{\epsilon \rightarrow 0} \mu_\epsilon(\mathbf{P},\mathbf{Q}) > 0.$$
As discussed in the introduction, the notion of stochastic stability has played a major role in the characterization of limiting behaviors of evolutionary processes, particularly regarding the evolution of convention (see the monograph \cite{Young_ISSS} and references therein).
Our aim is to characterize the set of stochastically stable states.

Towards this end, define a \textit{homogenous state} as a state in which all players use a single language, so that
$$(\mathbf{P},\mathbf{Q}) = \Big( (P,Q),...,(P,Q) \Big)$$
for some $(P,Q)\in\Lmn$.

Let $\mathcal{P}_{m,n,d}^{0}$ denote the dynamics under zero mutation probability, i.e., $\epsilon=0$. By construction, a homogenous state is absorbing. That is, once the dynamics reach a homogenous state---in the absence of mutations---it must remain in that state. Furthermore, it is possible to reach a homogenous state in one step from any non-homogenous state by all agents imitating the fittest strategy. Accordingly, the set of homogenous states are the recurrent communication classes of $\mathcal{P}_{m,n,d}^{0}$.

Some homogenous states result in higher average societal fitness than others. If a language $(P,Q)$ satisfies $\tr (PQ) = \min \{m,n\}$, then it is among the maximally efficient languages. We call such a language \textit{aligned}. Likewise, if $\tr (PQ) <\min \{m,n\}$, we say that $(P,Q)$ is an \textit{unaligned} language. Now define $\mathcal{O}$, the set of \textit{optimal} states, as the set of homogenous states with aligned languages. As intended, an optimal state maximizes average societal fitness, defined by
\begin{equation}\nonumber
W(\textbf{P},\textbf{Q}) = \frac{1}{N}\sum_{i=1}^N f_i(\textbf{P},\textbf{Q}).
\end{equation}

Our main result is that regardless of the choice of parameter $d\geq 1$, the set of stochastically stable states are precisely the set of optimal states.

\begin{theorem}\label{thm:mainm=n}
   A state $(\mathbf{P},\mathbf{Q})$ is stochastically stable if and only if it belongs to $\mathcal{O}$, the set of optimal states.
 \end{theorem}

The remainder of this section is devoted to the proof of Theorem~\ref{thm:mainm=n}. We will proceed following arguments outlined in the background appendix.

We first show that a suboptimal state cannot be a stochastically stable state by showing that every tree rooted in a suboptimal state can be rewired to a tree rooted in an optimal state with lower resistance. This construction in itself shows that the only states that can be stochastically stable are a subset of the optimal states. Then, using the rich set of symmetries of the dynamics, we show that indeed any optimal state is a stochastically stable state.

Let us first investigate the structure of the edges with resistance one in $\Pmn$. In words, we wish to show when a homogenous state $(P,Q)^N$ can transition to another homogenous state $(\tilde{P},\tilde{Q})$ with positive probability after only a single mutation. The transition need not occur in a single stage. At stage $t=0$, the initial condition is the homogenous state $(P,Q)^N$.
At stage $t=1$, there is a single mutation by a single agent. Afterwards, the mutation free dynamics $\mathcal{P}_{m,n,d}^{0}$ complete the transition so that at some time $T>1$, the state is
$(\tilde{P},\tilde{Q})$.
{
\begin{claim}\label{claim:onemutation}
   A homogenous state $i=(\mathbf{P},\mathbf{Q})=(P,Q)^N$ can transition to another homogenous state $j=(\tilde{\mathbf{P}},\tilde{\mathbf{Q}})=(\tilde{P},\tilde{Q})^N$ with one mutation if and only if $(\tilde{P},\tilde{Q})\in \rho((P,Q),d)$ and $(\tilde{P},\tilde{Q})$ satisfies
 \begin{align}\label{eqn:inequality}
   \tr(PQ)\leq \frac{1}{2}{\tr(P\tilde{Q}+\tilde{P}Q)}.
 \end{align}
 \end{claim}
 \begin{proof}
   Consider the state
   \[(\mathbf{P},\mathbf{Q})^{(1)}=((P,Q),\ldots,(P,Q),(\tilde{P},\tilde{Q})),\]
   which is resulted from the homogeneous state $(P,Q)^N$ by one mutation of a language $(P,Q)$ to $(\tilde{P},\tilde{Q})\in \rho((P,Q),d)$.

Then
$$\Delta =  f_1((\mathbf{P},\mathbf{Q})^{(1)})-f_N((\mathbf{P},\mathbf{Q})^{(1)}).$$
would be the difference between the fitness of the first agent, a user of $(P,Q)$,
and the last agent, the user of $(\tilde{P},\tilde{Q})$, in the state $(\mathbf{P},\mathbf{Q})^{(1)}$.

The state $(\mathbf{P},\mathbf{Q})=(P,Q)^N$ can transition to $(\tilde{\mathbf{P}},\tilde{\mathbf{Q}})=(\tilde{P},\tilde{Q})^N$ with one mutation if and only if $(\tilde{P},\tilde{Q})\in \rho((P,Q),d)$ and $\Delta^{(1)}\leq 0$ because, then there would be a positive chance that in the next time step, all the users of $(P,Q)$ will revise their strategies at the next time step and adapt the fittest language $(\tilde{P},\tilde{Q})$.

By definition, the fitness of the first agent is
\begin{align*}
f_1((\mathbf{P},& \mathbf{Q})^{(1)})= \frac{1}{N-1}\left(
2(N-2)\tr(PQ) + \tr(P\tilde{Q}) + \tr(\tilde{P}Q)\right).
\end{align*}
Likewise, the fitness of the last agent is
$$f_N((\mathbf{P},\mathbf{Q})^{(1)}) = \frac{1}{N-1}\left(
(N-1)\tr(\tilde{P}Q) + (N-1)\tr(P\tilde{Q})\right).$$
Straightforward arguments imply that
   \begin{align}\label{eqn:difference}
     \Delta     &=
     \frac{1}{N-1}\left((N-2)\tr(2PQ-(P\tilde{Q}+\tilde{P}Q))\right).
   \end{align}
Condition (\ref{eqn:inequality}) is equivalent to $\Delta\le 0$, and hence, the result follows.
\end{proof}}

 Next, we show that for any unaligned language $(P,Q)$ and any $d\geq 1$, there exists a language $(\tilde{P},\tilde{Q})\in\rho((P,Q),d)$ such that $\eqref{eqn:inequality}$ holds.

 \begin{claim}\label{claim:complement}
  Let $(P,Q)$ be a language with $\tr(PQ)<\min(m,n)$. For any $d\geq 1$, there exists $(\tilde{P},\tilde{Q})\in \rho((P,Q),d)$ such that \eqref{eqn:inequality} holds. Furthermore, $\tr(\tilde{P}\tilde{Q})=\tr(PQ)+1$.
\end{claim}
\begin{proof}
It is sufficient to prove the claim for $d=1$ as $\rho((P,Q),1)\subseteq \rho((P,Q),d)$ for any $d\geq 1$.

Let $A=\{i\in \{1,\ldots,m\} \mid \sum_{j}P_{ij}Q_{ji}=1\}$, and $B=\{j\in \{1,\ldots,m\} \mid P_{ij}=1, i\in A\}$. In other words, $A$ is the set of the objects that are contributing to $\tr(PQ)$ and $B$ is the set of symbols corresponding to the set of objects in $A$. Since, $P$ is a row stochastic and binary matrix, it follows that $|A|=|B|=\tr(PQ)<m$. Therefore, there exists an object $i'\in \{1,\ldots,m\}\setminus A$ and a symbol $j'\in \{1,\ldots,n\}\setminus B$. Now, define matrices $\tilde{P}$ and $\tilde{Q}$ as follows:
  \begin{align}\label{eqn:deftP}
    \tilde{P}_{ij}=\left\{
    \begin{array}{ll}
      P_{ij}&\mbox{if $i\not= i'$}\\
      1& \mbox{if $i=i'$ and $j=j'$}\\
      0& \mbox{if $i=i'$ and $j\not= j'$;}
    \end{array}
    \right.,
  \end{align}
   and similarly,
  \begin{align}\label{eqn:deftQ}
    \tilde{Q}_{ji}=\left\{
    \begin{array}{ll}
      Q_{ji}&\mbox{if $j\not=j'$},\\
      1& \mbox{if $j=j'$ and $i=i'$}\\
      0& \mbox{if $j=j'$ and $i\not= i'$.}
    \end{array}
    \right. .
  \end{align}
By construction, both  $\tilde{P}$ and $\tilde{Q}$ are binary and stochastic matrices. Also, $\tr(\tilde{P}\tilde{Q})=\tr(PQ)+1$. Since $\tilde{P}$ differs from $P$ at most at the $i$th row and $\tilde{Q}$ is different from $Q$ at most at the $j$th row, and since the matrices are binary, it follows that $\dham((P,Q),(\tilde{P},\tilde{Q}))\le 1$.

It remains to establish the inequality \eqref{eqn:inequality}. All of the columns of $Q$ contributing to $\tr(PQ)$ are equal to the corresponding columns of $Q'$ and hence, $\tr(PQ)\leq \tr(PQ')$. Similarity, all of the rows of $P$ contributing to $\tr(PQ)$ are equal to the corresponding rows of $P'$, and hence $\tr(PQ)\leq \tr(P'Q)$. Therefore $(\tilde{P},\tilde{Q})$ satisfies \eqref{eqn:inequality}.
\end{proof}

Now, we are ready to prove that the only stochastically stable states of $\Pz$ are the optimal states.
\begin{claim}\label{claim:forward}
  For any $d\geq 1$, a homogenous state $(P,Q)^N$ is stochastically stable in $\Pz$ only if $(P,Q)$ is an aligned language.
\end{claim}
\begin{proof}
Let $\mathcal{T}$ be a minimum resistance tree\footnote{Set the definition of resistance and minimum resistance rooted tree in the appendix.} rooted at $(P,Q)^N$, where $(P,Q)$ is an unaligned language. Then, by Claim~\ref{claim:complement}, there exists a language $(\tilde{P},\tilde{Q})\in \rho((P,Q),d)$ such that \eqref{eqn:inequality} holds and $\tr(\tilde{P}\tilde{Q})\geq \tr(PQ)+1$.

Accordingly, the path from $(\tilde{P},\tilde{Q})^N$ to $({P},{Q})^N$ in $T$ must have an edge
in which the trace between the associated language matrices first decreases below $\tr(PQ)$.
That is, there exists an edge starting from a homogeneous state $(P^{\rm h},Q^{\rm h})^N$ (``h'' for ``head'') and ending in the homogenous state $(P^{\rm t},Q^{\rm t})^N$ (``t'' for ``tail'') such that
$$\tr(P^{\rm h},Q^{\rm h}) \ge \tr(PQ) + 1\quad\mathrm{and}\quad
\tr(P^{\rm t},Q^{\rm t}) \le \tr(PQ).$$
By Claim~\ref{claim:onemutation}, this edge must have a resistance of at least two because the trace is strictly decreasing along this edge.

Now construct a tree by removing the edge emanating from $(P^{\rm h},Q^{\rm h})^N$
and adding an edge from $(P,Q)^N$ to $(\tilde{P},\tilde{Q})^N$. The result is a new tree rooted at
$(P^{\rm h},Q^{\rm h})^N$. The removed edge has a resistance of at least two. By Claim~\ref{claim:complement}, the added edge has a resistance of one. Accordingly, the stochastic potential (i.e., sum of resistances of all links) of the new tree is lower, and therefore $(P,Q)^N$ cannot be a stochastically stable state. \end{proof}

Figure~\ref{figtreeproof} presents a graphical illustration of the proof, in which an edge of higher resistance is removed in favor of an edge with lower resistance, based on Claims~\ref{claim:onemutation}--\ref{claim:complement}.

\begin{figure}
\begin{center}
\includegraphics[width=0.2\textwidth]{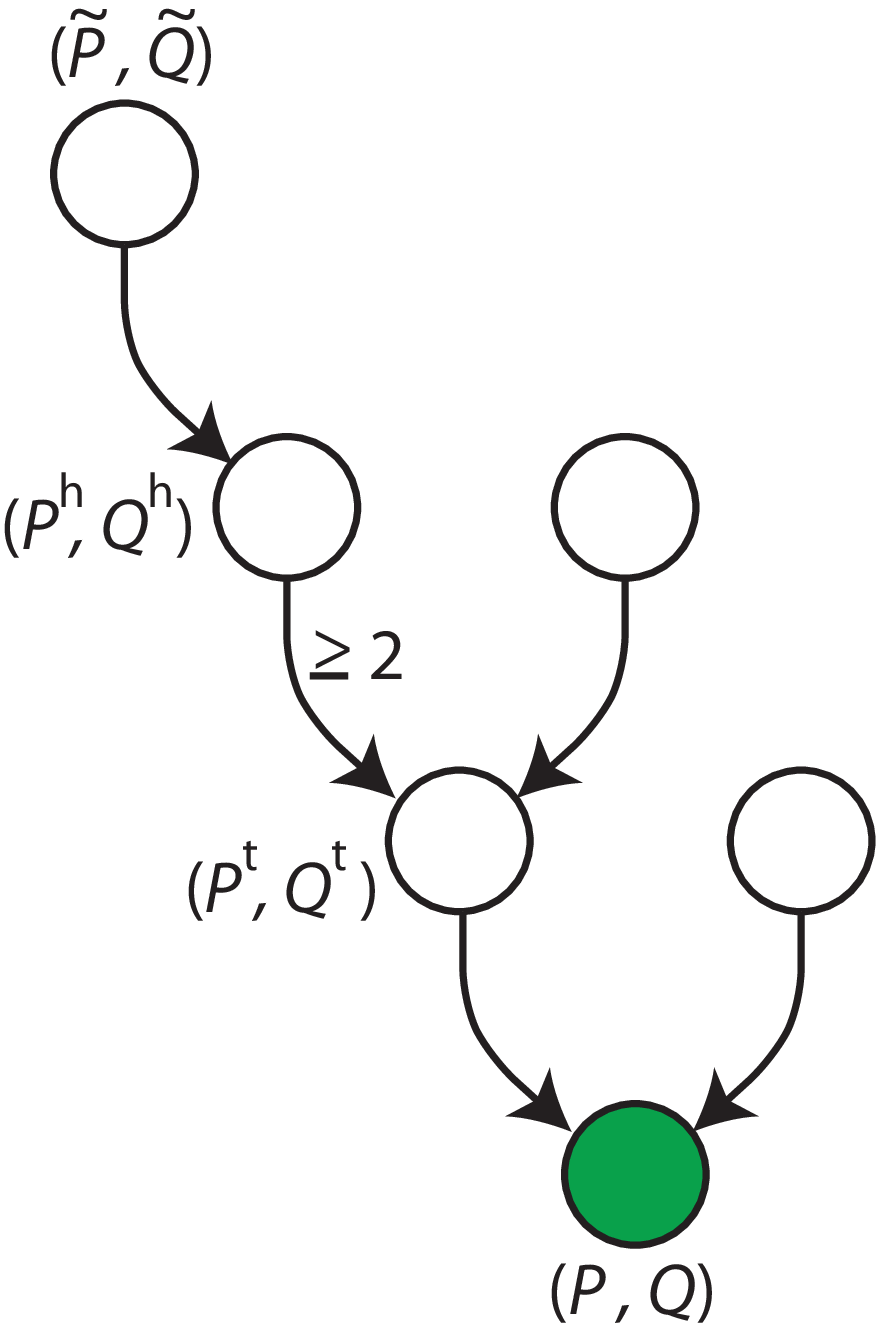}
\hfil
\includegraphics[width=0.2\textwidth]{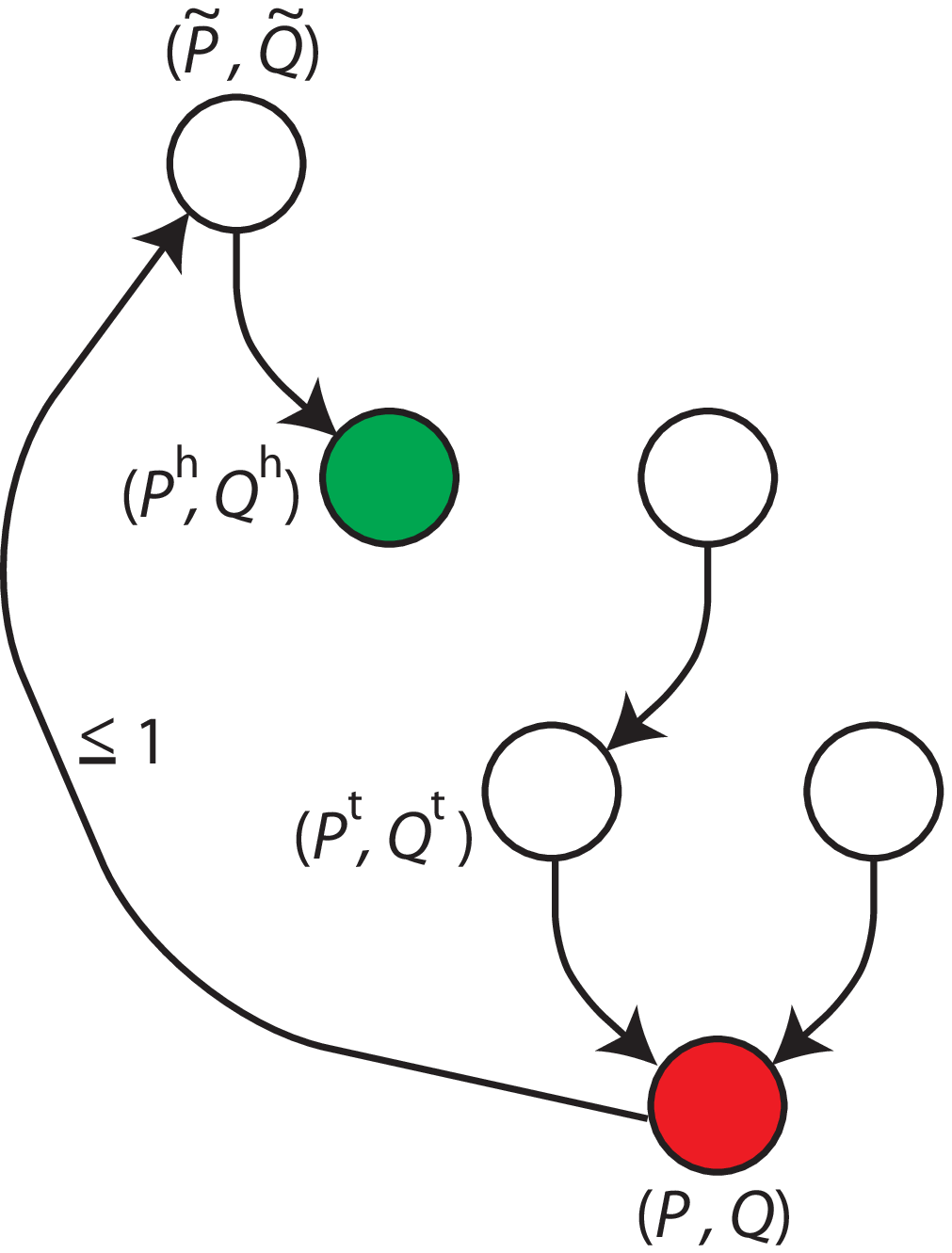}
\end{center}
\caption{Left: Original tree rooted at $(P,Q)^N$ with an edge of resistance at least two.
Right: Rewired tree rooted at $(P^{\rm h},Q^{\rm h})^N$ with a new edge of resistance at most one.}
\label{figtreeproof}
\end{figure}

Thus far, we have shown that the set of stochastically stable states is a subset of the optimal states. The next step is to show that the two sets are indeed equivalent. For this proof, we restrict our discussion to the case $m=n$. The proof of this assertion for the case $m\not=n$ follows the same lines, but with a more detailed argument.

We will argue that the set of optimal states is a subset of set of stochastically stable states
by making use of symmetries of $\Pmn$.

First, consider an arbitrary perturbed Markov chain $\mathcal{P}^\epsilon$ on the set of states $V$. Let us say that the perturbed Markov chain $\mathcal{P}^\epsilon$ admits an automorphism $\pi$ if $\pi:V\to V$ is a one to one mapping such that $r_{ij}=r_{\pi(i)\pi(j)}$ for any $i,j\in V$, where $r_{ij}$ is the resistance of the transition from state $i$ to state $j$. In fact, an automorphism $\pi$ is nothing but a relabeling of the states of $\mathcal{P}^\epsilon$ that preserves the resistance of each link. Let $\tau(\mathcal{P}^\epsilon)$ be the set of automorphisms of $\mathcal{P}^\epsilon$. Then, the following result follows immediately.

\begin{claim}\label{claim:symmetry}
 Let $S$ be the set of stochastically stable states of a perturbed Markov chain $\mathcal{P}^{\epsilon}$.  Then for any $i\in S$, we have
$$\{\pi(i)\mid \pi\in \tau(\mathcal{P}^\epsilon)\}\subseteq S.$$
\end{claim}
\begin{proof}
If $T$ is a rooted tree at $i$, then the image of $T$ under an automorphism $\pi$ is also a rooted tree at $\pi(i)$ which, by the definition of an automorphism, has the same weight as the original tree. Therefore, the assertion follows by Theorem~\ref{thm:young} in appendix.
\end{proof}

This claim states for general Markov chains that if a state, $i$, is stochastically stable, then the image $\pi(i)$ also is stochastically stable for any automorphism, $\pi$.

To complete the proof of Theorem~\ref{thm:mainm=n}, we will construct automorphisms between languages that simply consist of a relabeling of objects. For any $m\times m$ permutation matrix\footnote{An $m\times m$ binary matrix $P$ is a permutation matrix if every row and column of $P$ contains only one nonzero element.}  $\mathscr{P}$, let us define $\pi_{\mathscr{P}}:\Lmn\to\Lmn$ as follows:
\begin{align}\nonumber
  \pi_{\mathscr{P}}((P_1,Q_1),\ldots, (P_N,Q_N))=((\mathscr{P}P_1,Q_1\mathscr{P}^{-1}),\ldots,(\mathscr{P}P_N,Q_N\mathscr{P}^{-1})).
\end{align}
Since $\mathscr{P}$ is a permutation matrix, $\pi_{\mathscr{P}}$ is a bijection from $\Lmn\to \Lmn$. Also $\tr(PQ)=\tr(\mathscr{P}P Q\mathscr{P}^{-1})$ implies that $f_i(\bfP,\bfQ)=f_i(\pi_{\mathscr{P}}(\bfP,\bfQ))$ for all $i$. Finally, $\operatorname{rand}(\Lmn)$ has the same distribution as $\operatorname{rand}(\pi_{\mathscr{P}}(\Lmn))$. Therefore, it follows that $\pi_{\mathscr{P}}$ is an automorphism on $\Pmn$ for any permutation matrix $\mathscr{P}$.

Now consider a stochastically stable state $(\bfP,\bfQ)=(P,Q)^N$ with $m=n$. By Claim~\ref{claim:forward}, we need to have $\tr(PQ)=m$, and hence $Q=P^{-1}$.
Any aligned language $(\tilde{P},\tilde{P}^{-1})$ can be written as a $((\tilde{P}P^{-1})P,P^{-1}(\tilde{P}P^{-1})^{-1})$, i.e., they are related through the permutation $\mathscr{P} = \tilde{P}P^{-1}$. Therefore $\pi_{\tilde{P}P^{-1}}((P,P^{-1})^N)=(\tilde{P},\tilde{P}^{-1})^N$. From Claim~\ref{claim:symmetry}, it follows that any optimal state $(\tilde{P},\tilde{Q})^N$ is stochastically stable. This completes the proof of Theorem~\ref{thm:mainm=n}.

\subsection{Discussion}

Theorem~\ref{thm:mainm=n} implies that for sufficiently small mutation probability $\epsilon$, in the long run, the dynamic process $\mathcal{P}_{m,n,d}^{\epsilon}$ spends almost all of the time on the states that maximize linguistic coherence. The particular language nevertheless will change from time to time, consistent with observed phenomenon of linguistic drift \cite{Jespersen_PL}. In the more natural case of $m \neq n$, where the number of objects and symbols does not match, this drift ought to be particularly prevalent as the necessary ambiguity provides pathways for such changes. The many well documented cognates in modern natural languages point to the divergence in the meaning of a particular form as a vehicle for linguistic change \cite{P_diversity}.

In short, the intrinsic randomness that enables the players to search the set of languages will prevent settling into any sort of permanent language state. However, despite never freezing in a particular language, we can expect players to understand each other for a high proportion of the time.

{
\section{Localized competition dynamics}\label{sec:dynamic_var}
Up until this point we have studied the dynamics described by \eqref{eqn:maindynamic}
in which, in the absence of mutations, only the fittest players \textit{in the whole society} are able to reproduce themselves.
Inspired by pairwise comparison dynamics \cite{g1010003}, in this section we propose an intuitive new dynamic where in the absence of mutation, the fittest \textit{local} agents are able to reproduce.

In localized competition dynamics, each agent $i$ chooses a random subset of agents $\N_i[t]$ with $j\in \N_i[t]$ with probability $p_{ij}>0$ at every stage. Agent $i$ updates her
strategy according to
\begin{equation}
\label{dynamic_k_pc}
(P_i[t],Q_i[t]) = \begin{cases}
(P_{\hat{k}_i}[t],Q_{\hat{k}_i}[t]), & \text{with probability } 1-\epsilon \\
\operatorname{rand}(\mathcal{L}_{m,n}), & \text{with probability } \epsilon
\end{cases},
\end{equation}
where $\hat{k}_i\in \argmax_{j\in \N_i[t]}f_j(P[t],Q[t])$ is chosen uniformly from the set of agents with maximum fitness among the agents in $\N_i[t]$. In words, at each time instance, an agent observes a random subset of agents and adapts the fittest language among these agents with a high probability or mutates with a low probability. We can interpret (\ref{dynamic_k_pc}) as \textit{localized} competition over limited resources.

Without mutation (i.e.\ when $\epsilon=0$) a more fit player reproduces herself, while a less fit player dies out in long run. The fitness function itself still reflects a global interaction---effective communication with the broader population confers advantages in local competitions.

\begin{theorem}\label{cor:pc}
In localized competition dynamics, a state is stochastically stable if and only if it is contained in $\mathcal{O}$, the set of optimal states.
\end{theorem}

The proof of this result is similar to the proof of Theorem~\ref{thm:mainm=n}, and to avoid redundancy, we only provide a sketch of the proof.

To start with, the only recurrent classes of the unperturbed dynamics \eqref{dynamic_k_pc} ($\epsilon=0$) are the homogeneous states because for all pairs $(i,j)$, we have $p_{ij}>0$ which implies that there is a positive chance that $\N_i[t]=\{1,\ldots,N\}$ for all agents $i$, as a result, at each time instance there is a positive probability that all the agents in the society conform to one language (which is one of the fittest languages in the current state of society). Now, if at a given time instance $t$, all the agents in the society are using an unaligned language $(P,Q)$, following the same argument as of the proof of Claim \ref{claim:onemutation}, by one mutation to a language $(P',Q')$ satisfying \eqref{eqn:inequality}, all the agents will adapt $(P',Q')$ in finite time (almost surely). It is worth mentioning that the condition \eqref{eqn:inequality} does not depend on the size of the society $N$, which allows us to draw such a conclusion. Using an argument similar to that of Theorem~\ref{thm:mainm=n}, by proper rewiring of any routed tree routed at any non-optimal state, one can find a routed tree with a lesser weight that is routed at an optimal state which shows that the only stochastically stable states of the localized competition dynamics \eqref{dynamic_k_pc} are the optimal states. Finally, using the automorphisms of the underlying perturbed Markov chain (through permutation matrices), we can show that any optimal state is stochastically stable.}

\section{Simulations}

In this section, we provide simulation results of the mutation-selection dynamics \eqref{eqn:maindynamic} and the localized competition dynamics \eqref{dynamic_k_pc}.

 In Figure~\ref{fig:simul}, a simulation result of main dynamics \eqref{eqn:maindynamic} is provided for a society of size $N=30$ and $m=n=3$. There, the percentage of the population incorporating an aligned language as a function of time $t$ is plotted for $t=1,\ldots,300$. The initial language $(P_i[0],Q_i[0])$ of each agent is chosen uniformly randomly from the ensemble of binary stochastic matrices. As predicted by Theorem~\ref{thm:mainm=n}, after an initial transition phase, the dynamics spend most of its time on optimal states-the states with most of the agents using an aligned language. In this simulation, the language of each agents at the termination time $t=300$ is the aligned language
 \[(P,Q)=\left(\left[\begin{array}{ccc}
    0   &  1 &    0\\
    0   &  0 &    1\\
    1   &  0 &    0\\
 \end{array}\right],\left[\begin{array}{ccc}
    0   &  0 &    1\\
    1   &  0 &    0\\
    0   &  1 &    0\\
 \end{array}\right]\right).\]

 \begin{figure}[t]
\begin{center}
\includegraphics[width=0.7\textwidth]{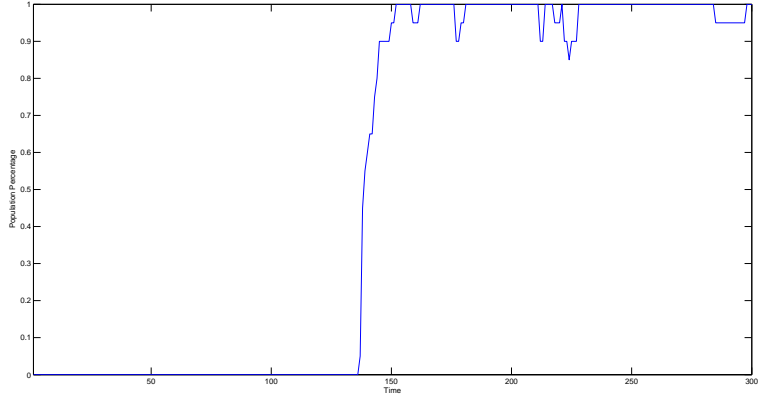}
\end{center}
\caption{A simulation of the dynamics \eqref{eqn:maindynamic} for a society with $N=20$ agents, $m=n=3$, revision probability $p=0.3$, mutation probability $\epsilon=0.01$, and mutation radius $d=3$.}
\label{fig:simul}
\end{figure}

A similar simulation for the localized competition dynamics \eqref{dynamic_k_pc} is provided in Figure~\ref{fig:local} with the same set of parameters. For this simulation, we chose the probability $p_{ij}$ of agent $j$ being a neighbor of agent $i$ uniformly randomly from $(0,1)$. Again as observed by the simulation, the language of the agents in the society drifts towards aligned languages as time passes. In this simulation, at the termination time $t=1000$, all the agents use the aligned language
 \[(P,Q)=\left(\left[\begin{array}{ccc}
    1   &  0 &    0\\
    0   &  1 &    0\\
    0   &  0 &    1\\
 \end{array}\right],\left[\begin{array}{ccc}
    1   &  0 &    0\\
    0   &  1 &    0\\
    0   &  0 &    1\\
 \end{array}\right]\right).\]

 \begin{figure}[t]
\begin{center}
\includegraphics[width=0.7\textwidth]{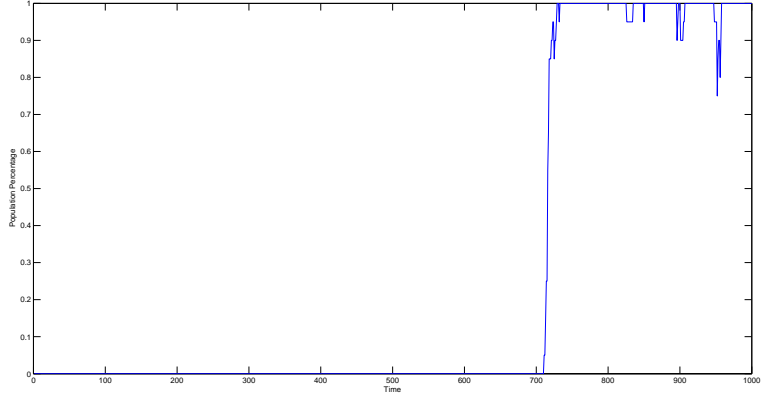}
\end{center}
\caption{A simulation of the localized competition dynamics \eqref{dynamic_k_pc} for a society with $N=20$ agents, $m=n=3$, revision probability $p=0.3$, and mutation probability $\epsilon=0.01$ with $p_{ij}$ being chosen uniformly from $(0,1)$.}
\label{fig:local}
\end{figure}

Finally, Figure~\ref{figdrift} illustrates linguistic drift of the main dynamics \eqref{eqn:maindynamic} through stochastic stability. There are $N=10$ agents, $m=n=3$, and the mutation probability is $\epsilon = 0.2$. There are only $6$ aligned languages. The figure shows the number of agents using each of these aligned languages. We see that the population experiences transitions between nearly homogeneous states of aligned languages. Note that in this simulation, the mutation probability is an order of magnitude larger than the prior simulations. For smaller mutation probabilities, there would be closer conformity to the main result of stochastic stability in Theorem~\ref{thm:mainm=n}. That is, the simulations would show transitions between homogenous states of aligned languages with a large occupation measure at such homogenous states. The selected mutation probability shows similar behavior with a much shorter simulation time.

\begin{figure}[ht]
\begin{center}
\includegraphics[width=0.7\textwidth]{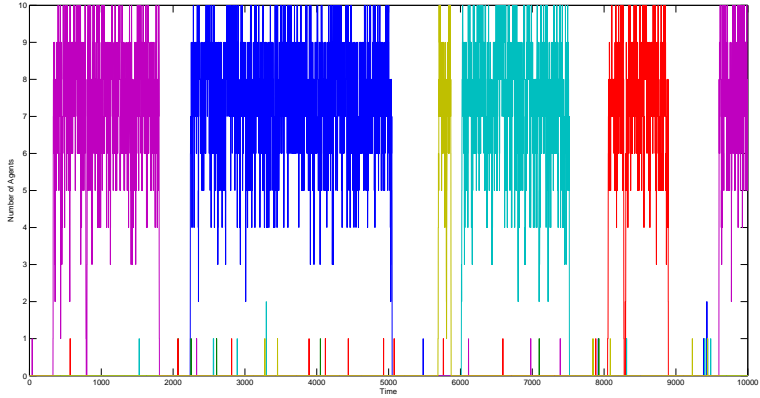}
\end{center}
\caption{A simulation of the dynamics \eqref{eqn:maindynamic} for a society with $N=10$ agents, $m=n=3$, revision probability $p=0.3$, mutation probability $\epsilon=0.2$, and mutation radius $d=3$.}
\label{figdrift}
\end{figure}

\section{Concluding Remarks}
In this paper, we studied language games in finite populations. We showed that, unlike the infinite population case, in long run, the only emerging languages from various selection-mutation processes are the languages with maximum communication efficiency. Following prior work on the evolution of conventions and other settings \cite{Young_Conventions,Young_ISSS}, the proof relied on resistance tree analysis of perturbed Markov chains.

Several questions are left open for future studies. Among them is the role of the connectivity of graph on the emerging stochastically stable states. More precisely, what is the characterization of stochastically stable strategies for language games on general (not necessarily fully connected) graphs? Also of interest is the study of the language games in connection with the models of opinion dynamics such as Hegselmann-Krause dynamics \cite{Krause:02}. For example, what happens if the members of the society only observe the efficiency of languages similar to their own languages? Some initial progress along these lines are reported in \cite{evolang_fox}.

\begin{appendices}
\section{Perturbed Markov Chains and Stochastic Stability}\label{appendix:ss}

Let $\mathcal{P}^{\epsilon}$ be the transition probability matrix of an irreducible and aperiodic time-homogeneous Markov chain\footnote{Strictly speaking, a (homogeneous) Markov chain is characterized by the initial occupation measure of the Markov chain and its transition probability matrix (kernel). In this discussion, we informally identify Markov chains by their probability transition matrices.} over a finite set of states $Z$ for each $\epsilon \in (0,\bar{\epsilon}]$. If for each $z,z' \in Z$ we have
\begin{equation}\nonumber
\lim_{\epsilon \rightarrow 0} \mathcal{P}^{\epsilon}_{z,z'} = \mathcal{P}^{0}_{z,z'},
\end{equation}
for some Markov chain $\mathcal{P}^{0}$ over $Z$, and
\begin{equation}\nonumber
\label{resistance}
0 < \lim_{\epsilon \rightarrow 0} \frac{\mathcal{P}^{\epsilon}_{z,z'}}{\epsilon^{r(z,z')}} < \infty,
\end{equation}
for some $r(z,z') \geq 0$ then $\mathcal{P}^{\epsilon}$ is a \textit{regular perturbed Markov process}. We call $\mathcal{P}^{0}$ the unperturbed process. If $\mathcal{P}^{\epsilon}_{z,z'} = 0$ for all $\epsilon$, then we define $r(z,z') = \infty$. It is straightforward to see that $\mathcal{P}_{m,n,d}^{\epsilon}$ is a regular perturbed Markov process, with $\mathcal{P}_{m,n,d}^{0}$ being the reducible Markov chain obtained by substituting $\epsilon = 0$.

Note that since $\mathcal{P}^\epsilon$ is an aperiodic and irreducible Markov chain for any $\epsilon>0$, it has a unique stationary distribution $\mu(\mathcal{P}^{\epsilon})$. However, it is possible that the limiting Markov chain $\mathcal{P}^0$ is not an irreducible or aperiodic chain and hence, $\mathcal{P}^0$ may not admit a unique stationary distribution. The notion of stochastic stability is concerned with which one of the stationary distributions survive. A state $z \in Z$ is \textit{stochastically stable} if
\begin{equation}\nonumber
\lim_{\epsilon \rightarrow 0} \mu_z(\mathcal{P}^{\epsilon}) > 0.
\end{equation}
To characterize the stochastically stable states of a perturbed Markov chain, we will make use of the theory of resistance trees \cite{Young_Conventions}. Let $R_1,...,R_J \subset Z$ be the recurrent communication classes of $\mathcal{P}^{0}$. Given two recurrent communication classes $R_i$ and $R_j$, let $\{z_0,z_1,...,z_K\}$ be a path satisfying $z_0 \in R_i$ and $Z_K \in R_j$. We call the quantity $\sum_{k = 0}^{K-1} r(z_k,z_{k+1})$ the \textit{resistance} of the path. With slight abuse of notation we define $r_{ij}$ to be the \textit{least resistance} among all such paths.

Consider a graph $G$ whose vertex set is the set of recurrent communication classes. An $R_i$-tree $T$ is a spanning tree in $G$ such that for any vertex $R_j, j \neq i$ there is a unique directed path from $R_j$ to $R_i$. We define
\begin{equation}\nonumber
\gamma(R_i) = \min_{T \in \mathcal{T}_{R_i}} \sum_{(R_j,R_k) \in E(T)} r_{jk},
\end{equation}
where $\mathcal{T}_{R_i}$ is the set of all $R_i$ trees in $G$, and $E(T)$ is the set of edges of $T$. $\gamma(R_i)$ is referred to as the \textit{stochastic potential} of $R_i$ \cite{Young_Conventions}.
 The following theorem \cite{Young_Conventions} characterizes exactly the set of stochastically stable states.
\begin{theorem}\label{thm:young} (\cite{Young_Conventions}, Lemma 1)
Let $\mathcal{P}^{\epsilon}$ be a regular perturbed Markov process and let $R_1,..,R_J$ be the recurrent communication classes of the unperturbed process $\mathcal{P}^{0}$. Then the stochastically stable states are precisely those states contained in the recurrent communication classes with minimum stochastic potential.
\end{theorem}
\end{appendices}

\small
\bibliographystyle{plain}
\bibliography{signaling}

\end{document}